\documentclass[10pt,reqno]{amsart}

\usepackage[T1]{fontenc}
\usepackage{lmodern}
\usepackage{microtype}
\usepackage{mathtools}
\usepackage{amssymb}
\usepackage{graphicx}
\usepackage{enumitem}
\usepackage{placeins}
\usepackage[colorlinks=true,linkcolor=blue,citecolor=blue,urlcolor=blue]{hyperref}

\hypersetup{
  pdftitle={From 12 to 6: Sharpening the Three-Charge Bound in Maxwell's Problem},
  pdfauthor={Andrei Gabrielov, Dmitry Novikov, Tomer Novikov, Boris Shapiro}
}

\numberwithin{equation}{section}
\allowdisplaybreaks
\theoremstyle{plain}
\newtheorem{theorem}{Theorem}[section]
\newtheorem{proposition}[theorem]{Proposition}
\newtheorem{lemma}[theorem]{Lemma}

\theoremstyle{remark}

\newcommand{\R}{\mathbb{R}}
\newcommand{\C}{\mathbb{C}}
\newcommand{\Qfield}{\mathbb{Q}}
\newcommand{\dd}{\,\mathrm{d}}
\newcommand{\NP}{\operatorname{NP}}
\newcommand{\Res}{\operatorname{Res}}
\newcommand{\Vol}{\operatorname{Vol}}
\newcommand{\Qpos}{\R_{>0}^{2}}
\newcommand{\restr}[2]{#1\big|_{#2}}

\title[From 12 to 6]{From 12 to 6: Sharpening the Three-Charge Bound in Maxwell's Problem}

\author[A. Gabrielov]{Andrei Gabrielov}
\address{Department of Mathematics, Purdue University, West Lafayette, Indiana 47907, USA}

\author[D. Novikov]{Dmitry Novikov}
\address{Department of Mathematics, Weizmann Institute of Science, Rehovot, Israel}

\author[T. Novikov]{Tomer Novikov}
\address{Department of Economics, Northwestern University, Evanston, Illinois 60208, USA}

\author[B. Shapiro]{Boris Shapiro}
\address{Department of Mathematics, Stockholm University, SE-106 91 Stockholm, Sweden}

\date{}


\begin{document}

\begin{abstract}
In \cite{GNS} we proved that, for every $\alpha>0$, the potential of
three positive point charges has at most $12$ nondegenerate equilibrium
points.  We also observed that the same method would give the sharper
bound $6$ if a certain auxiliary polynomial system $Q=R=0$ had at least
four solutions, counted with multiplicity, in each open quadrant of the
$(f,g)$-plane.  Here we prove this four-solution statement.  The main new
ingredient is a separation argument at the unique saddle point of a
separated-variable first integral.  Consequently, the upper bound for three charges improves from $12$ to
$6$.
\end{abstract}

\maketitle

\section{Introduction}\label{sec:introduction}

In \S~113 of his treatise \cite{Ma}, J.~C.~Maxwell discussed the
maximal number of isolated equilibrium points generated by point charges
in $\mathbb R^3$ and asserted the bound $(k-1)^2$ for $k$ charges.  The
argument supporting this assertion is incomplete; see the appendix to
\cite{GNS}.\footnote{For historical discussion of contemporary criticism
of inconsistencies and incomplete arguments in Maxwell's treatise, including
comments by J.~J.~Thomson, J.~C.~McConnel, H.~Poincar\'e, and O.~Heaviside,
see \cite[Chapter~6]{Wa}.} We refer to the problem of bounding the number
of equilibrium points of a finite configuration of point charges as
\emph{Maxwell's problem}.

A general finiteness bound for a natural generalization of Maxwell's problem
was established in Theorem~1.5(a) of \cite{GNS} and subsequently improved,
for example, in \cite{Zo}.  Very recently Arathoon, Ball, and Kvalheim
\cite{ABK}, building on ideas from \cite{Edel}, subsequently constructed five positive
charges in $\mathbb R^3$ whose Coulomb potential has at least $24$
nondegenerate critical points.  Thus Maxwell's proposed bound is false in
general.

For three positive charges, Theorem~1.5(b) of \cite{GNS} gives the upper
bound $12$, whereas Maxwell's proposed bound is $4$.  Our goal is to improve
$12$ to $6$.  For three equal charges, the bound $4$ is known \cite{Ts}.

Our set-up is as follows. We may assume that the three charges are noncollinear since the collinear situation is trivial.  In the noncollinear case, place three positive
charges $\zeta_1,\zeta_2,\zeta_3=1$ at $(0,0)$, $(1,0)$, and $(a,b)$,
respectively, with $b\neq0$.  After normalizing the third charge, the
potential has the form
\begin{equation}\label{eq:potential}
  V_{\alpha}=\sum_{i=1}^{3}\zeta_i\rho_i^{-\alpha},
  \qquad
  \rho_i=r_i^2,
  \qquad
  \alpha>0.
\end{equation}
Every critical point lies in the affine plane of the charges; in the
noncollinear case it lies in the open triangle that they span.  As
recalled in \cite{GNS}, Morse-theoretic considerations give the parity
relation
\begin{equation}\label{eq:parity}
  \#\{\text{equilibria}\}=2m_0+2,
\end{equation}
where $m_0$ is the number of planar local minima.  In particular, the
number of nondegenerate equilibria is even.  The main result of this note is as follows. 

\begin{theorem}\label{thm:main}
For every $n\geq2$ and every $\alpha>0$, a potential generated by three
positive point charges in $\R^n$ has at most $6$ nondegenerate equilibrium
points.
\end{theorem}

The proof retains the two Rolle--Khovanskii reductions and the
mixed-volume count from \cite{GNS}.  The only new step is a four-contact
lemma for a compact component of the auxiliary contact curve.  We first
recall the reduction and isolate the statement that must be proved.

\section{Reduction of the three-charge problem}\label{sec:reduction}

Following \cite{GNS}, introduce
\begin{equation}\label{eq:fg}
  f=\left(\frac{\rho_1}{\rho_2}\right)^{\alpha+1},
  \qquad
  g=\left(\frac{\rho_1}{\rho_3}\right)^{\alpha+1}.
\end{equation}
Critical points of $V_{\alpha}$ correspond bijectively to intersections
in the positive quadrant $\Qpos$ of the curves
\begin{align}
  \gamma_1
    &=\left\{
      f^{1/(\alpha+1)}\frac{\xi_2}{\xi_1}=1
      \right\},
      \label{eq:gamma1}\\
  \gamma_2
    &=\left\{
      g^{-1/(\alpha+1)}\xi_2
      =f^{-1/(\alpha+1)}\xi_3
      \right\},
      \label{eq:gamma2}
\end{align}
where
\begin{align}
  \xi_1(f,g)
    &=(ag+\zeta_2f)^2+b^2g^2,
    \label{eq:xi1}\\
  \xi_2(g)
    &=((a-1)g-\zeta_1)^2+b^2g^2,
    \label{eq:xi2}\\
  \xi_3(f)
    &=((a-1)\zeta_2f+a\zeta_1)^2
      +b^2(\zeta_2f+\zeta_1)^2.
    \label{eq:xi3}
\end{align}
The form $\xi_1$ is positive definite, and $\xi_2,\xi_3$ are strictly
positive quadratic polynomials on the real line.

The curves $\gamma_1$ and $\gamma_2$ are integral curves of closed
logarithmic one-forms
\begin{equation}\label{eq:eta12}
  \eta_1=\dd\log\Xi,
  \qquad
  \eta_2=\dd\log\Psi,
\end{equation}
where $\Xi$ and $\Psi$ are defining functions for \eqref{eq:gamma1} and
\eqref{eq:gamma2}.  They are separating solutions in Khovanskii's sense,
so the Rolle--Khovanskii theorem applies twice; see
\cite{GNS,Khovanskii}.

Let
\begin{equation}\label{eq:N}
  N=\#(\gamma_1\cap\gamma_2\cap\Qpos).
\end{equation}
For generic parameters, all intersections considered below are isolated and
are counted with their local intersection multiplicities.  The first
Rolle--Khovanskii step gives
\begin{equation}\label{eq:first-rolle}
  N\leq N_1+N_2,
\end{equation}
where $N_1=2$ is the number of unbounded components of $\gamma_2$ in
$\Qpos$ and
\begin{equation}\label{eq:N2}
  N_2=\#(\gamma_2\cap\Gamma\cap\Qpos).
\end{equation}
Here $\Gamma=\{Q=0\}$ is the contact curve defined by
\begin{equation}\label{eq:def-Q}
  Q\,\dd f\wedge\dd g
  =fg\,\xi_1\xi_2\xi_3\,\eta_1\wedge\eta_2.
\end{equation}
Its Newton polygon is
\begin{equation}\label{eq:NPQ}
  \NP(Q)=\{(p,q):2\leq p+q\leq6,\ 0\leq p,q\leq4\}.
\end{equation}

The second Rolle--Khovanskii step gives
\begin{equation}\label{eq:second-rolle}
  N_2\leq N_3+N_4,
\end{equation}
where $N_3=0$ and
\begin{equation}\label{eq:N4}
  N_4=\#\bigl(\Gamma\cap\{R=0\}\cap\Qpos\bigr).
\end{equation}
The equality $N_3=0$ follows because $\Gamma$ has no unbounded
component and its only point on the coordinate axes is the origin, which
is isolated in $\Gamma$.  The polynomial $R$ is defined by
\begin{equation}\label{eq:def-R}
  R\,\dd f\wedge\dd g
  =fg\,\xi_2\xi_3\,\dd Q\wedge\eta_2.
\end{equation}

After replacing $R$ by a reduction $\widetilde R=R-qQ$, the
Bernstein--Kushnirenko theorem gives
\begin{equation}\label{eq:BKK}
  \#\{Q=\widetilde R=0\text{ in }(\C^*)^2\}
  \leq
  2\Vol\bigl(\NP(Q),\NP(\widetilde R)\bigr)
  =28.
\end{equation}
The count is taken with intersection multiplicity; see
\cite{Bernstein,GNS}.  Two nonreal conjugate common zeros of
$\xi_1,\xi_2,\xi_3$ contribute multiplicity at least $6$ each.  The proof
in \cite{GNS} also finds at least two real solutions of $Q=R=0$ in each
of the three open quadrants different from $\Qpos$.  Hence
\begin{equation}\label{eq:old-bound}
  N_4\leq28-12-6=10,
  \qquad
  N\leq2+0+10=12.
\end{equation}

The improvement rests on the following statement.

\begin{proposition}[Four-solution proposition]\label{prop:four-solutions}
Fix an open quadrant
\[
  \mathcal Q_{\sigma}=I\times J,
  \qquad
  I,J\in\{(0,\infty),(-\infty,0)\}.
\]
Let $p$ be the unique zero of $\eta_2$ in this quadrant, whose existence
is established in Lemma~\ref{lem:saddle}, and let $O$ be the component of
$\Gamma\cap\mathcal Q_{\sigma}$ containing $p$.  Assume that $Q=R=0$
has finitely many solutions in $(\C^*)^2$ and that $\Gamma$ is
nonsingular along $O$.  Then the system $Q=R=0$ has total intersection
multiplicity at least $4$ in $\mathcal Q_{\sigma}$.
\end{proposition}

Assuming Proposition~\ref{prop:four-solutions} for generic parameters,
the three quadrants outside $\Qpos$ contribute at least $12$ to the
left-hand side of \eqref{eq:BKK}.  Therefore
\begin{equation}\label{eq:new-N4}
  N_4\leq28-12-12=4,
\end{equation}
and \eqref{eq:first-rolle}--\eqref{eq:second-rolle} give
\begin{equation}\label{eq:generic-six}
  N\leq2+0+4=6.
\end{equation}
Sections~\ref{sec:saddle} and \ref{sec:oval} prove
Proposition~\ref{prop:four-solutions}; Section~\ref{sec:genericity}
justifies the genericity assumptions and completes the proof of
Theorem~\ref{thm:main}.

\section{A separated saddle in every quadrant}\label{sec:saddle}

The one-form $\eta_2$ separates in the variables $f$ and $g$:
\begin{equation}\label{eq:eta2-separated}
  \eta_2=P(f)\,\dd f+S(g)\,\dd g,
\end{equation}
where
\begin{equation}\label{eq:P-S}
  P(f)=-\frac{1}{(\alpha+1)f}+\frac{\xi_3'(f)}{\xi_3(f)},
  \qquad
  S(g)=\frac{1}{(\alpha+1)g}-\frac{\xi_2'(g)}{\xi_2(g)}.
\end{equation}
Write
\[
  \xi_3(f)=Af^2+Bf+C,
  \qquad
  \xi_2(g)=\widehat A g^2+\widehat B g+\widehat C,
\]
where $A,C,\widehat A,\widehat C>0$.  Clearing denominators in
\eqref{eq:P-S} gives
\begin{equation}\label{eq:pi-sigma-quotients}
  P(f)=\frac{\pi(f)}{(\alpha+1)f\xi_3(f)},
  \qquad
  S(g)=-\frac{\sigma(g)}{(\alpha+1)g\xi_2(g)},
\end{equation}
with
\begin{align}
  \pi(f)
    &=(\alpha+1)f\xi_3'(f)-\xi_3(f)
      =(2\alpha+1)Af^2+\alpha Bf-C,
      \label{eq:pi}\\
  \sigma(g)
    &=(\alpha+1)g\xi_2'(g)-\xi_2(g)
      =(2\alpha+1)\widehat A g^2+\alpha\widehat B g-\widehat C.
      \label{eq:sigma}
\end{align}
The product of the two roots of $\pi$ is
$-C/((2\alpha+1)A)<0$.  Thus $\pi$ has one simple positive root and one
simple negative root.  The same argument applies to $\sigma$.

\begin{lemma}\label{lem:saddle}
On every open quadrant $I\times J$, the form $\eta_2$ has a
single-valued real-analytic primitive
\begin{equation}\label{eq:Phi}
  \Phi(f,g)=u(f)+s(g).
\end{equation}
This primitive has exactly one critical point $p=(f_0,g_0)$ in the
quadrant, and $p$ is a nondegenerate saddle.  Moreover, $u$ is a
one-well function: it decreases strictly to $f_0$, increases strictly
after $f_0$, and tends to $+\infty$ at both ends of $I$.  Similarly,
$s$ is a one-hill function and tends to $-\infty$ at both ends of $J$.
\end{lemma}

\begin{proof}
The functions $f\xi_3(f)$ and $g\xi_2(g)$ do not vanish on $I$ and $J$.
Hence $P$ and $S$ admit single-valued real-analytic primitives on these
intervals.

Let $f_-<0<f_+$ be the roots of $\pi$.  On $(0,\infty)$, the denominator
of $P$ is positive, so $P<0$ on $(0,f_+)$ and $P>0$ on $(f_+,\infty)$.
On $(-\infty,0)$, the sign of the factor $f$ reverses the sign of the
denominator; because $\pi$ is positive on $(-\infty,f_-)$ and negative
on $(f_-,0)$, one again has $P<0$ before the root and $P>0$ after it.
Consequently, $u$ decreases and then increases on either interval.
Furthermore,
\begin{align*}
  P(f)&\sim-\frac{1}{(\alpha+1)f}
       &&\text{as }f\to0,\\
  P(f)&\sim\frac{2\alpha+1}{(\alpha+1)f}
       &&\text{as }|f|\to\infty.
\end{align*}
After integration, these asymptotics show that $u(f)\to+\infty$ at both
ends of $I$.

The same sign analysis, applied to the second expression in
\eqref{eq:pi-sigma-quotients}, shows that $s$ first increases and then
decreases; the analogous asymptotics give $s(g)\to-\infty$ at both ends
of $J$.  Hence the unique zeros $f_0$ of $P$ and $g_0$ of $S$ give
the unique critical point of $\Phi$ in the quadrant.

At that point,
\begin{equation}\label{eq:second-derivatives}
  u''(f_0)
  =\frac{\pi'(f_0)}{(\alpha+1)f_0\xi_3(f_0)}>0,
  \qquad
  s''(g_0)<0.
\end{equation}
Indeed, $\pi'(f_0)$ and $f_0$ have the same sign, and the corresponding
statement for $\sigma$ gives the second inequality.  Therefore
\begin{equation}\label{eq:Hessian}
  \operatorname{Hess}_p\Phi
  =\begin{pmatrix}
      u''(f_0)&0\\
      0&s''(g_0)
    \end{pmatrix}
\end{equation}
has signature $(+,-)$.
\end{proof}

Set
\begin{equation}\label{eq:c0H}
  c_0=\Phi(p)=u(f_0)+s(g_0),
  \qquad
  H=\operatorname{diag}\bigl(u''(f_0),s''(g_0)\bigr).
\end{equation}

\begin{lemma}[Separation lemma]\label{lem:separation}
The set $\{\Phi>c_0\}\cap(I\times J)$ has exactly two connected
components,
\begin{equation}\label{eq:LR}
  L=\{f<f_0\}\cap\{\Phi>c_0\},
  \qquad
  \mathcal R=\{f>f_0\}\cap\{\Phi>c_0\}.
\end{equation}
In particular, no connected subset of $\{\Phi>c_0\}$ meets both sides of
the line $\{f=f_0\}$.  Symmetrically, $\{\Phi<c_0\}$ has exactly two
connected components separated by $\{g=g_0\}$.
\end{lemma}

\begin{proof}
On the line $f=f_0$,
\begin{equation}\label{eq:line-separation}
  \Phi(f_0,g)=u(f_0)+s(g)
  \leq u(f_0)+s(g_0)=c_0,
\end{equation}
with equality only at $g=g_0$.  Thus $\{\Phi>c_0\}$ does not meet
$\{f=f_0\}$, and the sets in \eqref{eq:LR} are disjoint open sets whose
union is $\{\Phi>c_0\}$.  Both are nonempty because $u$ tends to
$+\infty$ on either side of $f_0$.

We verify that $L$ is connected.  For fixed $g$, its horizontal slice is
an interval adjacent to the left endpoint of $I$:
\begin{equation}\label{eq:beta}
  \{f\in I:f<f_0,\ u(f)>c_0-s(g)\}
  =\{f\in I:f<\beta(g)\},
\end{equation}
where
\begin{equation}\label{eq:beta-def}
  \beta(g)
  =\left(\restr{u}{I\cap(-\infty,f_0)}\right)^{-1}
     \bigl(c_0-s(g)\bigr).
\end{equation}
The function $\beta$ is continuous.  Given two points of $L$, the
$g$-coordinates between them form a compact interval.  Choose an
$f$-coordinate lying strictly to the left of the minimum of $\beta$ on
that interval.  Moving horizontally to this coordinate, then vertically,
and finally horizontally to the second point produces a path in $L$.
Thus $L$ is connected.  The same argument applies to $\mathcal R$.
Interchanging the roles of the one-well function $u$ and the one-hill
function $s$ proves the assertion for $\{\Phi<c_0\}$.
\end{proof}

\section{The oval and its four contacts}\label{sec:oval}

The following identity is proved in \cite{GNS}:
\begin{equation}\label{eq:GNS-identity}
  Q=-\frac{1+2\alpha}{(\alpha+1)^2}\,\xi_1\xi_2\xi_3-fgQ_1,
\end{equation}
where $\NP(fgQ_1)$ lies strictly inside
$\NP(\xi_1\xi_2\xi_3)$.  It follows that $\Gamma=\{Q=0\}$ has no unbounded component, that
$Q<0$ on each coordinate axis away from the origin, and that the origin is
an isolated point of $\Gamma$.  Since
$\eta_2(p)=0$, equation \eqref{eq:def-Q} also gives $Q(p)=0$.

\begin{lemma}\label{lem:oval}
The component $O$ defined in Proposition~\ref{prop:four-solutions} is
compact and contained in the open quadrant.  Under the assumptions of
Proposition~\ref{prop:four-solutions}, $O$ is a smooth closed curve,
$\restr{\Phi}{O}$ is nonconstant, and the zeros of $R$ on $O$ are exactly
the critical points of $\restr{\Phi}{O}$, with the same multiplicities.
More precisely, if $\gamma$ is a regular parametrization of $O$, then
\begin{equation}\label{eq:R-on-oval}
  R(\gamma(t))=\rho(t)(\Phi\circ\gamma)'(t)
\end{equation}
for a nowhere-vanishing analytic function $\rho$.
\end{lemma}

\begin{proof}
The absence of unbounded components makes $O$ bounded.  It cannot
accumulate at a nonzero point of a coordinate axis because $Q$ is
strictly negative there, and it cannot accumulate at the origin because
the origin is an isolated zero of $Q$.  Hence $O$ is compact and lies in
the open quadrant.  The regularity assumption makes it a compact
connected one-dimensional manifold without boundary, and therefore a
smooth closed curve.  If $\restr{\Phi}{O}$ were constant, then
\eqref{eq:R-on-oval} below would imply that $Q$ and $R$ share the whole
curve $O$, contrary to the finiteness assumption.

Let $\gamma$ be a regular analytic parametrization of $O$, and choose a
unit normal field $N$ along it.  Evaluating \eqref{eq:def-R} on
$(\gamma'(t),N(t))$ gives
\begin{equation}\label{eq:wedge-on-oval}
 R(\gamma(t))\,\dd f\wedge\dd g(\gamma'(t),N(t))
 =-fg\xi_2\xi_3\,\dd Q(N(t))\,(\Phi\circ\gamma)'(t),
\end{equation}
up to the fixed orientation sign.  The factors
$\dd f\wedge\dd g(\gamma',N)=\pm\lvert\gamma'\rvert$,
$\dd Q(N)=\pm\lvert\nabla Q\rvert$, and $fg\xi_2\xi_3$ are all
nonzero.  This proves \eqref{eq:R-on-oval}, including equality of the
vanishing orders.
\end{proof}

The saddle $p\in O$ is a critical point of $\restr{\Phi}{O}$ because
$\dd\Phi(p)=0$.  If $v_0$ is a nonzero tangent vector to $O$ at $p$, then
$p$ is a simple zero of $(\Phi\circ\gamma)'$ precisely when
\begin{equation}\label{eq:tangent-Hessian}
  v_0^{\mathsf T}Hv_0\neq0.
\end{equation}

\begin{lemma}[Four-contact lemma]\label{lem:four-contacts}
The total multiplicity of the zeros of $R$ on $O$ is at least $4$.
\end{lemma}

\begin{proof}
The function $h=\restr{\Phi}{O}$ is a nonconstant real-analytic function
on a circle.  Its critical points are isolated.  A zero of odd order of
$h'$ is exactly a point at which $h'$ changes sign; hence the number of
odd-order zeros of $h'$ around the circle is even.  Moreover, $h$ attains
its minimum and maximum at two distinct critical points.

Suppose, for a contradiction, that the total multiplicity of the zeros
of $h'$ is at most $3$.  Since $p$ is one of these zeros and there are at
least two distinct critical points, only the following possibilities can
occur:
\begin{enumerate}[label=\textup{(\roman*)},leftmargin=2.4em]
  \item exactly two critical points, both simple;
  \item exactly two critical points, of total multiplicity $3$;
  \item exactly three critical points, all simple.
\end{enumerate}
Cases \textup{(ii)} and \textup{(iii)} are impossible because each gives
an odd number of odd-order zeros of $h'$.

It remains to exclude case \textup{(i)}.  The two critical points are the
unique global minimum and maximum of $h$.  First suppose that $p$ is the
minimum.  Then
\begin{equation}\label{eq:Phi-above}
  \Phi>c_0\qquad\text{on }O\setminus\{p\}.
\end{equation}
The punctured circle $O\setminus\{p\}$ is connected, so the separation
lemma implies that it is contained entirely in $L$ or entirely in
$\mathcal R$.

Because $p$ is a simple critical point of $h$ and a local minimum,
\begin{equation}\label{eq:positive-tangent}
  v_0^{\mathsf T}Hv_0>0.
\end{equation}
The positive cone of the quadratic form
$H=\operatorname{diag}(+,-)$ contains no vertical vector.  Thus the
$f$-component of $v_0$ is nonzero.  The two ends of the arc
$O\setminus\{p\}$ approach $p$ with opposite tangent directions $v_0$
and $-v_0$, and therefore with opposite signs of $f-f_0$.  One end lies
in $L$ and the other in $\mathcal R$, a contradiction.

If $p$ is the maximum, the same argument uses the two components of
$\{\Phi<c_0\}$ separated by the line $g=g_0$; in this case the negative
cone of $H$ contains no horizontal vector.  Hence case
\textup{(i)} is also impossible, and the total multiplicity is at least
$4$.
\end{proof}

At a smooth point of $\Gamma$, the local intersection multiplicity of
$\{Q=0\}$ and $\{R=0\}$ equals the vanishing order of $R$ along $O$.
Lemmas~\ref{lem:oval} and \ref{lem:four-contacts} therefore prove
Proposition~\ref{prop:four-solutions}.

\section{Genericity and completion of the proof}\label{sec:genericity}

We now verify that the hypotheses of
Proposition~\ref{prop:four-solutions} hold on a dense open set of
parameters.  Put
\begin{equation}\label{eq:theta}
  \theta=(a,b,\zeta_1,\zeta_2,\alpha).
\end{equation}
After multiplication by harmless powers of $\alpha+1$, the coefficients
of $Q$ and $R$ are polynomial functions of $\theta$.

\begin{proposition}\label{prop:genericity}
There is a nonempty Zariski-open subset $U$ of the complex parameter space
with the following properties:
\begin{enumerate}[label=\textup{(G\arabic*)},leftmargin=3.2em]
  \item\label{G1}
  for every $\theta\in U$, the system $Q=R=0$ has finitely many
  solutions in $(\C^*)^2$;
  \item\label{G2}
  for every admissible real $\theta\in U$, the curve $\Gamma$ is
  nonsingular in $(\R^*)^2$.
\end{enumerate}
Moreover, $U$ meets the admissible real parameter set
$b\neq0$, $\zeta_1,\zeta_2,\alpha>0$ in a Euclidean-dense subset.
\end{proposition}

\begin{proof}
For \ref{G1}, note that, for the fixed supports of $Q$ and $R$, the pairs
having a nonconstant common factor form a Zariski-closed subset of the
coefficient space.  Indeed, this set is the finite union, over the possible
bidegrees of a common factor, of the projective images of the corresponding
multiplication maps.  At
\begin{equation}\label{eq:witness}
  \theta_*
  =\left(\frac{3}{10},\frac{7}{5},\frac{7}{10},
          \frac{14}{5},\frac12\right),
\end{equation}
an exact computation over $\Qfield$ gives $\gcd(Q,R)=1$ in
$\Qfield[f,g]$.  Thus \ref{G1} holds on a nonempty Zariski-open set.

For \ref{G2}, set
\begin{equation}\label{eq:r2r4}
  r_2(f)=\Res_g(Q,Q_g),
  \qquad
  r_4(f)=\Res_g(Q_f,Q_g).
\end{equation}
At a singular point of $\Gamma$ in $(\C^*)^2$, the $f$-coordinate is a
common zero of $r_2$ and $r_4$.  The factor $f$ is a structural common
factor, coming from the isolated zero at the origin and the support
condition \eqref{eq:NPQ}.  At each of the two common zeros of $\xi_1,\xi_2,\xi_3$ recalled in
Section~\ref{sec:reduction}, the curve $\Gamma$ is singular; this is part
of the multiplicity-six analysis in \cite{GNS}.  The two corresponding
$f$-coordinates are precisely the roots of $\xi_3$.  Hence $\xi_3$ is
also a structural common factor.  Thus
\begin{equation}\label{eq:structural-gcd}
  f\xi_3(f)\mid\gcd(r_2,r_4).
\end{equation}
At the witness \eqref{eq:witness}, exact arithmetic gives
\begin{equation}\label{eq:witness-degrees}
  \deg_f r_2=24,
  \qquad
  \deg_f r_4=21,
  \qquad
  \gcd(r_2,r_4)=f\xi_3(f)
\end{equation}
up to a nonzero constant.  The leading coefficients responsible for the
two displayed degrees are nonzero at the witness.  Hence these degrees
remain constant on a Zariski-open neighborhood of the witness.
The subresultant criterion, together with
\eqref{eq:structural-gcd}, then shows that
\begin{equation}\label{eq:generic-gcd}
  \gcd(r_2,r_4)=f\xi_3(f)
\end{equation}
throughout a nonempty Zariski-open set.

For admissible real parameters, $\xi_3(f)>0$ for every real $f$, and a
point of $(\R^*)^2$ has $f\neq0$.  Thus \eqref{eq:generic-gcd} rules out
real singular points of $\Gamma$ in the torus, proving \ref{G2}.  Since
the witness \eqref{eq:witness} is admissible and the complement of $U$ is
a proper algebraic subset, $U$ intersects the admissible real parameter
set in a Euclidean-dense subset.
\end{proof}

The exact calculations at \eqref{eq:witness} are short and use only the
definitions \eqref{eq:xi1}--\eqref{eq:def-R}.  A reproducible
exact-arithmetic script, \texttt{verify\_genericity.py}, accompanies the
source of this article.

\begin{proof}[Proof of Theorem~\ref{thm:main}]
For parameters in the dense open set of
Proposition~\ref{prop:genericity}, Proposition~\ref{prop:four-solutions}
applies in every open quadrant.  The argument in
\eqref{eq:new-N4}--\eqref{eq:generic-six} therefore gives the bound $6$
for generic noncollinear configurations.

Now let an arbitrary admissible potential have $k$ nondegenerate
equilibrium points.  By the implicit function theorem, all $k$ points
persist under sufficiently small perturbations of the positions, the
charges, and, if necessary, $\alpha$.  The perturbed parameters may be
chosen noncollinear and in the dense open set of
Proposition~\ref{prop:genericity}, while preserving positivity of the
charges and the inequality $\alpha>0$.  The perturbed
configuration has at least $k$ equilibrium points and at most $6$, so
$k\leq6$.  Finally, write a point as $x=y+z$, where $y$ belongs to the affine
span of the charges and $z$ is orthogonal to it.  The normal component of
$\nabla V_\alpha(x)$ is a nonzero scalar multiple of
$z\sum_i\zeta_i\rho_i^{-\alpha-1}$; hence it vanishes only when $z=0$.
Thus every equilibrium point lies in the affine span of the charges, and
the problem reduces to dimension at most $2$.
\end{proof}

The proof is independent of numerical experimentation.
Figure~\ref{fig:four-equilibria} illustrates the positive-quadrant
geometry but is not used in the argument.

\section{Sharpness of the method and computational checks}\label{sec:checks}

\begin{figure}[!t]
  \centering
  \includegraphics[width=0.66\textwidth]{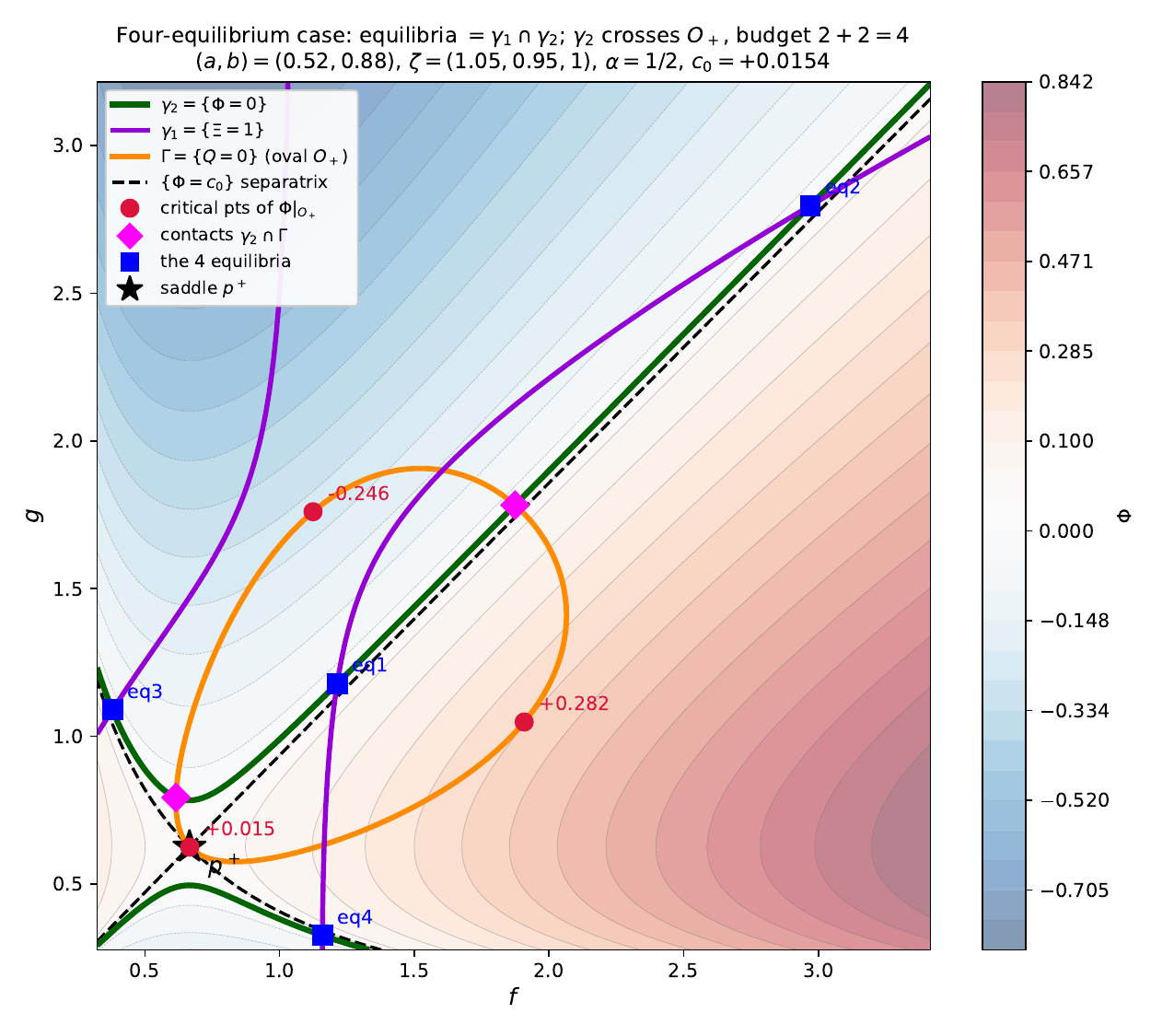}
  \caption{A positive-quadrant configuration with four equilibria.  The
  orange curve is the oval $O^+\subset\Gamma$; the green and purple curves
  are $\gamma_2$ and $\gamma_1$, respectively; and the dashed curve is the
  separatrix level $\Phi=c_0$.  Red points mark critical points of
  $\restr{\Phi}{O^+}$, magenta diamonds mark intersections
  $\gamma_2\cap\Gamma$, blue squares mark the four equilibria
  $\gamma_1\cap\gamma_2$, and the black star marks the saddle $p^+$.}
  \label{fig:four-equilibria}
\end{figure}
\FloatBarrier

We record the following computational checks because they clarify where
the two-step Rolle--Khovanskii method loses information.  Exact rational constructions
of $Q$ and $R$, exact resultants, high-precision
root isolation, and independent tracing of the real ovals were
performed for ten parameter sets.  These examples cover
$\alpha\in[1/20,20]$, flat, tall, and obtuse charge triangles, and charge
ratios up to $100:1$.  In every tested example, the unique zero of
$\eta_2$ in each quadrant is a nondegenerate saddle, and $Q=R=0$ has
exactly four real solutions in each open quadrant.  Thus there are
$16$ real torus solutions.  In the same examples the mixed-volume count
is saturated:
\begin{equation}\label{eq:saturation}
  16+2\cdot6=28.
\end{equation}
The relevant Newton polygon areas are $12$ and $15$, while the area of
their Minkowski sum is $55$, giving the mixed volume $28$.

The remaining slack lies in the Rolle--Khovanskii inequalities.  The
positive-quadrant oval $O^+$ carries four critical points of
$\restr{\Phi}{O^+}$.  Depending on the level defining $\gamma_2$, the
curve $\gamma_2$ meets $O^+$ in $0$, $2$, or $4$ points.

All steps in the count are generically sharp except for the combined
Rolle estimates
\begin{equation}\label{eq:remaining-slack}
  N\leq2+\#(\gamma_2\cap\Gamma\cap\Qpos).
\end{equation}
When the level of $\gamma_2$ lies between the two middle critical values
of $\restr{\Phi}{O^+}$, the curve has four crossings with the oval.  To
reach Maxwell's conjectural bound $4$ by this route, one would have to
show that in precisely this situation the first Rolle step overcounts by
at least $2$.

Combining Theorem~\ref{thm:main} with \eqref{eq:parity} gives
$m_0\leq2$.  Hence the number of nondegenerate equilibria is $2$, $4$, or
$6$.  An improvement to $m_0\leq1$ would be equivalent to Maxwell's
conjectured bound for three charges.

\section{Final remarks}

\noindent
{\bf 1.} Edelsbrunner, Fillmore, and Oliveira \cite{Edel} found a counterexample
to Conjecture~1.8 of \cite{GNS}, which asserted that, for every
$\alpha>0$, the number of nondegenerate equilibrium points of a finite
configuration of positive charges does not exceed the number of its
effective Voronoi cells.  One of the main results of \cite{GNS} proves
this assertion for every fixed configuration and all sufficiently large
$\alpha$.  Thus it remains possible that some configuration of three
charges has exactly six equilibrium points for special positive values of
$\alpha$.

\noindent
{\bf 2.} In view of a recent counterexample \cite{ABK} the question of finding  new conjectures/results for the maximal number of points of equilibrium for configurations of point charges has resurfaced anew. 

\section*{Acknowledgments}

Parts of this work, including the saddle-separation argument and the symbolic and numerical verification, were developed with the assistance of Claude (Anthropic). All results were independently verified by the authors. 

This project  was mainly carried out while D.~Novikov was visiting the
Institute for Advanced Study.  He thanks the Institute for its 
hospitality and excellent working conditions.  D.~Novikov was supported
by the Kovner Member Fund, Israel Science Foundation grant 1167/17, and
Minerva grant 714141.

\end{document}